\documentclass{article}
\usepackage[titletoc, title]{appendix}

\usepackage{hyperref}
\usepackage[english]{babel}

\usepackage{graphicx}
\usepackage{amssymb}
\usepackage{amsthm, amsmath,amssymb}
\usepackage[all]{xy}
\usepackage{graphicx}
\usepackage{psfrag}
\usepackage{multirow}
\usepackage{tikz}
\usepackage{color}

\usepackage[applemac]{inputenc}

\usepackage{multicol}

\usepackage {tikz}
\usepackage{tkz-berge}
\usetikzlibrary {positioning}

\usetikzlibrary{arrows,shapes}

\usetikzlibrary{decorations.markings}
\usetikzlibrary{plotmarks}
\usepackage{pgfplots}

\definecolor {processblue}{cmyk}{0.96,0,0,0}

\def\Rb{\mathbb{R}}		

\newcommand{\rk}{\operatorname{rank }}
\newcommand{\im}{\operatorname{Im}}

\newcommand{\restr}[1]
   {\vrule height1ex width.4pt
depth1.4ex\lower1.4ex\hbox{\scriptsize $\,#1$}}

\newtheorem{theorem}{Theorem}
\newtheorem{corollary}[theorem]{Corollary}
\newtheorem{proposition}{Proposition}

\newtheorem{lemma}{Lemma}

\theoremstyle{definition}
\newtheorem{definition}{Definition}
\newtheorem{remark}{Remark}

\newtheorem{example}{Example}

\author{Inês Cruz\thanks{Centro de Matem\'atica da Universidade do Porto (CMUP), Departamento de Matem\'atica,  Faculdade de Ci\^encias da Universidade do Porto, R. Campo Alegre, 687, 4169-007 Porto, Portugal. E-mails: imcruz@fc.up.pt, mmmatos@fc.up.pt.},\,\, Helena Mena-Matos\footnotemark[1]\,\, and M. Esmeralda Sousa-Dias\thanks{Center for Mathematical Analysis, Geometry and Dynamical Systems (CAMGSD),
Departamento de Matem\'atica, Instituto Superior T\'ecnico, Av. Rovisco Pais, 1049-001 Lisboa, Portugal. E-mail: e.sousa.dias@tecnico.ulisboa.pt.}}

\title {Multiple reductions, foliations and the dynamics of cluster maps} 
\begin{document}
\maketitle

\begin{abstract}
Reduction of cluster maps via presymplectic and Poisson structures is described in terms of the canonical foliations defined by these structures. In the case where multiple reductions coexist, we establish conditions on the underlying presymplectic and Poisson structures that allow for an interplay between the respective foliations. It is also shown how this interplay may be explored to  simplify the analysis and obtain an effective geometric description of the dynamics of the original (not reduced) map. Examples illustrating several features of this approach are presented.
\end{abstract}
\medskip



\section{Introduction}

Cluster maps are birational maps defined via mutations, which are the main operations in the  theory of cluster algebras \cite{FoZe}, \cite{Ze}. More precisely, cluster maps are maps  associated to quivers (oriented graphs) satisfying a mutation-periodicity property, which give rise to  recurrences with the Laurent property. The study of these maps has been the subject of some recent works and can be found  for instance in \cite{FoHo2}, \cite{FoHo1}, \cite{InHeEs} and \cite{InHeEs2}.

Presymplectic and Poisson structures were introduced in the theory of cluster algebras in \cite{GeShVa2} and \cite{FoGo} (see also the monograph \cite{GeShVa} and references therein) and, among other applications, they were used to reduce cluster maps to symplectic maps and  to study their integrability (\cite{Fo}, \cite{FoHo1} and \cite{FoHo2}). The presymplectic and Poisson structures considered in that theory are meant to deal with mutations and are of a particular type. They are known, in cluster algebra theory, as \emph{log-canonical} structures, since they are constant in logarithmic coordinates. Only this type of presymplectic and Poisson structures will be considered in the present work. 

The most effective study of cluster maps is precisely the one that makes use of  log-canonical structures, and is achieved through reduction to spaces of lower dimension.   By {\it reduction} of a given map $\varphi$ we mean the existence of a  map $\psi$ and of a submersion $\pi$ onto a lower dimensional space, such that $\pi\circ \varphi=\psi\circ\pi$. The map $\psi$ will be called a {\it reduced map of $\varphi$} and the pair $(\psi,\pi)$ will be called a {\it reduced system}.

Reduction of cluster maps via presymplectic and Poisson structures was obtained by Fordy and Hone in \cite{FoHo1} for maps associated to mutation-periodic quivers of period 1 and by Cruz and Sousa-Dias \cite{InEs2} for mutation-periodic quivers with arbitrary period. Presymplectic reduction is always possible as long as the skew-symmetric matrix $B$ defining the quiver is singular and produces a reduced map defined on a space whose dimension is equal to the rank of $B$. 
 
On the contrary, reduction via  Poisson structures may fail to exist due to the lack of a nontrivial Poisson structure for which the cluster map is a Poisson map.  Moreover, it may happen that there are different Poisson structures leading to reductions of the same cluster map to spaces of different dimensions (and equal to the corank of the skew-symmetric matrices defining the respective Poisson structures). 

In this work we explore the coexistence of multiple reductions (presymplectic and Poisson) and its consequences to the dynamics of a cluster map. Each reduction process gives rise to two kind of objects: (a) a foliation of the domain of the cluster map defined by the fibres of a submersion $\pi$; (b) a reduced system $(\psi, \pi)$. When multiple reductions exist, we show that there are conditions on the underlying presymplectic and Poisson structures  which guarantee that they define a \emph{flag of foliations}. The consequences  of the existence of such flag of foliations  to the dynamics of the cluster map is further explored, showing how to use it to simplify  the study of the map and obtain a better geometric description of its dynamics. 

Reduction of cluster maps via presymplectic forms has been widely applied and often leads to integrable (reduced) maps \cite{FoHo2}, \cite{FoHo1}. It has also played an important role in the study of the original map (see \cite{Ho} and \cite{HoSw}), more precisely, in obtaining the general solution for the associated recurrence. The existence of multiple reductions via Poisson structures has not yet witnessed the same level of development in the study of cluster maps, the study performed here is one of the possible directions   in the use of (compatible) Poisson structures in the dynamics of cluster maps. 

The structure of the paper is as follows. Section~\ref{back} is devoted to background material on cluster maps. 
  In Section~\ref{red}, we consider log-canonical presymplectic and Poisson structures and describe, in terms of submersions, the null foliation and the symplectic foliation  determined by these structures. We also  give necessary and sufficient conditions for a null foliation to be a simple subfoliation of a symplectic foliation (and vice-versa) and for a symplectic foliation to be a simple subfoliation of another symplectic foliation. 
 Section~\ref{din} explores the consequences to the dynamics of cluster maps of the existence of multiple reductions, in particular how   the dynamics of the multiple reduced systems  allows us to draw conclusions on the dynamics of the original map. The last section is devoted to examples illustrating the main results of the paper: the Somos-5 recurrence and a particular instance of the Somos-7 recurrence which fits into the cluster algebra setting.
 \section{Preliminaries on cluster maps}\label{back}

The notion of mutation-periodic quiver was introduced in \cite{FoMa} and gives rise to the definition of {\it  cluster map} as follows.  A quiver is an oriented graph with $N$ nodes and  with (possibly multiple) arrows  between the nodes. If the quiver has no loops nor 2-cycles, it is represented by an $N\times N$ skew-symmetric matrix $B=[b_{ij}]$ whose positive entries $b_{ij}$ denote the number of arrows from node $i$ to node $j$. These are the only quivers we consider and they will be denoted by $Q_B$.

To each node $i$ of a quiver $Q_B$ one associates a variable $x_i$, called {\it cluster variable}. The $N$-tuple $\mathbf{x}=(x_1, \ldots, x_N)$ is known as the {\it initial cluster} and the pair $(B,\mathbf{x})$ is called the {\it initial seed}.

The {\it mutation} $\mu_k$ at the node $k$ of  a quiver, is  an operation that transforms a seed $(B,\mathbf{x})$ into the seed $(B',\mathbf{x'})$ as follows:
\begin{itemize}
\item $\mu_k(B)= B'=\left[b'_{ij}\right]$ with
\[ b'_{ij}=\begin{cases}
-b_{ij},&\text{$ (k-i)(j-k)=0$}\\
b_{ij}+\frac{1}{2} \left(|b_{ik}|b_{kj}+ b_{ik}|b_{kj}|\right),&\text{otherwise}.
\end{cases}
\]
\item $\mu_k(x_1, \ldots, x_N)=\mathbf{x'}=(x_1,\ldots,x_{k-1}, x'_k,x_{k+1}, \ldots, x_N)$, with
\[x'_k=\frac{\prod_{j: b_{kj}\geq0}x_j^{b_{kj}}+\prod_{j: b_{kj}\leq0}x_j^{-b_{kj}}}{x_k}.
\]
\end{itemize}

A quiver $Q_B$ is said to be mutation-periodic if there exists a positive integer $m$ such that
\[ \mu_m\circ\cdots\circ\mu_1(B)= \sigma^{-m} B \sigma^m. \]
where $\sigma$ is the permutation $\sigma: (1,2,\ldots,N)\mapsto (2,3,\ldots,N,1)$. The mutation-period of   $Q_B$ is the smallest integer $m$ for which the above identity holds.

A mutation-periodic quiver of period $m$ gives rise to a system of $m$ recurrence relations of order $N$ whose solutions correspond to the orbits (under iteration) of the following map:
\[
 \varphi(\mathbf{x})=\sigma^m\circ\mu_m\circ\cdots\circ\mu_1 (\mathbf{x}). 
\]
This map is called the {\it cluster map} associated to the quiver.

\begin{example}\label{ex1}  

If $r$ and $s$ are non-negative integers, the matrix
\[
B=\begin{bmatrix}
0&r&-1&-1&s\\
-r&0&r+s&r-1&-1\\
1&-r-s&0&r+s&-1\\
1&1-r&-r-s&0&r\\
-s&1&1&-r&0
\end{bmatrix}
\]

\noindent defines a 1-periodic quiver if $r=s$ and a 2-periodic quiver otherwise. The cluster map is
\begin{equation}\label{somos5}
\varphi(x_1, \ldots, x_5)= \left(x_2,x_3,x_4,x_5, \frac{x_2^r x_5^r+x_3x_4}{x_1}\right)
\end{equation}
in the 1-periodic case (i.e. $r=s$) and 
\small
\begin{equation}\label{map5nos2per}
\varphi(x_1,x_2,x_3,x_4,x_5) =  \left(x_3,x_4,x_5,\frac{x_2^rx_5^s+x_3x_4}{x_1},\frac{x_3^s(x_2^rx_5^s+x_3x_4)^r+x_1^rx_4 x_5}{x_1^r x_2} \right)
\end{equation}
\normalsize
in the 2-periodic case.

The cluster map \eqref{somos5}  is associated to the recurrence relation  of order 5
\[x_{n+5}x_n= x_{n+1}^r x_{n+4}^r+x_{n+2}x_{n+3},\]
which, when $r=1$, is the well-known Somos-5 recurrence relation. On the other hand, the cluster map \eqref{map5nos2per} is associated to the following system  of two recurrence relations of order 5
\[\left\{\begin{aligned}
x_{2n+4}x_{2n-1}&= x_{2n}^r x_{2n+3}^s+x_{2n+1}x_{2n+2}\\
x_{2n+5}x_{2n}&= x_{2n+1}^s x_{2n+4}^r+x_{2n+2}x_{2n+3}.
\end{aligned}\right.
\]
We refer to \cite{FoMa} and \cite{InEs2} for details.
\end{example}
Because each mutation is an involution, every cluster map is a birational map, that is, rational with rational inverse. 
As our interest in cluster maps lies on their dynamics, we will consider their domain of definition to be  $\Rb_+^N$ (the set of points where all coordinates are strictly positive) which guarantees that any of its iterates is well defined. 

Reduction of a cluster map $\varphi$ by means of a presymplectic structure was proved in \cite{FoHo1} for 1-periodic quivers (defined in $\mathbb{C}^N$) and in \cite{InEs2} for quivers with arbitrary period. It relies on the existence of a presymplectic structure which is invariant under $\varphi$, as proved in the same references.   In fact, the log-canonical presymplectic form
\begin{equation}
\label{omega}
\omega = \sum_{1\leq i<j\leq N} \frac{b_{ij}}{x_ix_j} dx_i  \wedge dx_j, 
\end{equation}
where $B$ is the matrix defining the quiver associated to $\varphi$, is invariant under $\varphi$.  

The existence of a submersion $\widehat\pi:\Rb_+^N \rightarrow \Rb_+^{2k}$, with $2k=\rk B<N$, and of a (symplectic) reduced map $\widehat\varphi:\Rb_+^{2k} \rightarrow \Rb_+^{2k}$ such that $\widehat\pi\circ\varphi=\widehat\varphi \circ \widehat\pi$, follows from Darboux's theorem \cite{LiMa} applied to the presymplectic form \eqref{omega}.

In what concerns the reduction of a cluster map $\varphi$ by means of  a log-canonical Poisson structure $P$, this is possible if there exists such a (nontrivial) Poisson structure which is invariant under $\varphi$.  More precisely, if there exists a Poisson structure of the form
\begin{equation}\label{PP}
P=\sum_{1\leq i<j\leq N} c_{ij} x_ix_j \frac{\partial}{\partial x_i}\wedge\frac{\partial}{\partial x_j},
\end{equation}
satisfying $\varphi_*P=P$, where $C=[c_{ij}]$ is an integer skew-symmetric matrix with nontrivial kernel (here $\varphi_*$ denotes the pushforward by $\varphi$).
Examples where such reduction is possible can be found for example in \cite{FoHo1} and \cite{InEs2}.
 
Again, reduction by  a Poisson structure translates  into the existence of a submersion $\widetilde\pi:\Rb_+^N \rightarrow \Rb_+^{s}$, where   $s=\dim \ker C$,  and of a (reduced) map $\widetilde\varphi:\Rb_+^{s} \rightarrow \Rb_+^{s}$ such that $\widetilde\pi\circ\varphi=\widetilde\varphi \circ \widetilde\pi$  (see \cite[Theorem 5.1]{InEs2}).

Throughout the next sections, $B=[b_{ij}]$ will denote an integer skew-symmetric $N\times N$ matrix representing a mutation-periodic quiver $Q_B$ with period $m$ and $\varphi = \sigma^m\circ\mu_m \circ \cdots \circ \mu_1$  the associated cluster map defined on $\Rb_+^N$. Both  the presymplectic form $\omega$ in \eqref{omega}  and Poisson structures   of the form \eqref{PP} will be restricted to $\Rb^N_+$. Also,  unless otherwise stated, we will  assume that $\rk B=2 k<N$ and  $\dim \ker C= s >0$.

\section{Reductions of cluster maps and associated  foliations}\label{red}
In this section we characterise the foliations associated to reductions of a cluster map $\varphi$ via the presymplectic form $\omega$ in \eqref{omega} and via Poisson structures $P$ of the form \eqref{PP} which are invariant under $\varphi$.

Presymplectic manifolds and Poisson manifolds are natural generalizations of symplectic manifolds.  Each of these manifolds has its natural foliation, the \emph{null foliation} in the case of a presymplectic manifold and the \emph{symplectic foliation} in the case of a Poisson manifold.  Although these foliations are usually defined using integrable distributions  (see for example \cite{LiMa}, \cite{Bur} and references therein) and are in general \emph{singular} foliations (that is, distinct leaves can have different dimensions) we will show that, in our setting, each of these foliations   is  not only regular but also given by a surjective submersion $\pi$ with connected fibres.  In particular, this means that each leaf $L_\alpha$ of the foliation is a fibre of $\pi$: $L_\alpha=\pi^{-1}(\alpha)$. Following the terminology in \cite{MoMr} we will call these foliations {\it strictly simple foliations}.

\subsection{The null foliation and reduction via a presymplectic form}

The next proposition shows that the null foliation determined by the presymplectic form $\omega$ in \eqref{omega} is a strictly simple foliation. Its leaves are known as \emph{null leaves}.

\begin{proposition}\label{prop1} Let $\omega$ be the presymplectic form defined by the matrix $B$ as in \eqref{omega}   with $\rk B=2k$. Then, there exists a   surjective submersion $\widehat\pi: \Rb_+^N \rightarrow \Rb_+^{2k}$  with  connected fibres $N_{\alpha}=\widehat\pi^{-1}(\alpha)$ which determines the null foliation ${\cal F}^\omega$ of $\Rb_+^N$, that is
\[{\cal F}^\omega=\bigsqcup_{\alpha} N_{\alpha},\]  
with $\omega\restr{N_\alpha}=0$. Moreover the leaves $N_\alpha$ are algebraic varieties. 
\end{proposition}
\begin{proof} Let $\Phi:  \Rb_+^N \rightarrow  \Rb^N$ be the diffeomorphism given by 
\begin{equation}
\label{phi}
\Phi (x_1, \ldots, x_N)=(\ln x_1, \ldots, \ln x_N).
\end{equation}
The presymplectic form $\omega'=\left (\Phi^{-1}\right )^*\omega$ is the constant presymplectic form on $\Rb^N$, given by
\[\omega'= \sum_{1\leq i<j\leq N} b_{ij} dv_i  \wedge dv_j.\]
By \'Elie Cartan's theorem (see \cite{LiMa}), there exist $2k$ independent linear functions on $\Rb^N$ of the form $f_i({\bf v})= {\bf u}_i\cdot {\bf v}$ with ${\bf u}_i \in \im B$, such that $\omega'$ is written as 
\[\omega'= \sum_{m=1}^k  df_{2m-1}\wedge df_{2m}.\]
Now consider the foliation of $\Rb^N$ defined by the submersion  $\widehat\Pi:\Rb^N  \longrightarrow \Rb^{2k}$ given by
\[\widehat\Pi(\mathbf{v})= (f_1(\mathbf{v}), \ldots, f_{2k}(\mathbf{v})).
\]
The fibre of $\widehat\Pi$ through ${\bf v_0}$ is  $N' = \{\mathbf{v}\in \Rb^N:\,\, \widehat\Pi(\mathbf{v})=\widehat\Pi(\mathbf{v_0})\}$, which coincides with the affine subspace 
\[\mathbf{v_0} + \left ( \im B\right )^\perp.\] 
In particular: (a) each fibre of $\widehat\Pi$ is connected and therefore constitutes a leaf of the foliation; (b) this foliation depends only on $\im B$ and not on the particular choice of the linear functions $f_1, \ldots , f_{2k}$. 

In terms of the functions $y_i({\bf x}) = {\bf x}^{{\bf u}_ i}$, the expression of $\omega$ is given by:
\begin{equation}
\label{wnormal}
\omega = \sum_{m=1}^k  \frac{dy_{2m-1}}{y_{2m-1}}\wedge \frac{dy_{2m}}{y_{2m}}.
\end{equation}

Moreover, each fibre $N_\alpha$ of the surjective submersion $\widehat\pi :\Rb_+^N  \longrightarrow \Rb_+^{2k}$,
\begin{equation}
\label{pi}
\widehat\pi: \mathbf{x}  \longmapsto (y_1(\mathbf{x}), \ldots, y_{2k}(\mathbf{x}))
\end{equation}
is connected, since it is $\Phi$-diffeomorphic to a connected fibre $N'$ of $\widehat\Pi$. 

Note that $N_{\alpha} = \{\mathbf{x}\in \Rb_+^N:\,\, \widehat\pi (\mathbf{x})=\alpha\}$ and so $dy_i$ vanishes on $TN_\alpha$. Hence  the identity \eqref{wnormal} leads to the conclusion that $\omega$ vanishes on $N_\alpha$.

Finally, as $B$ is an integer matrix then each $f_i$ can be chosen to have  rational coefficients, which implies that each $y_i$ can be chosen to be a Laurent monomial. Thus, each fibre $N_\alpha$ is an algebraic variety. 
\end{proof}

As the presymplectic form $\omega$ in \eqref{omega} is invariant under the cluster map $\varphi$ defined by $B$  (see \cite[Theorem 3.1]{InEs2}), it turns out that  the submersion $\widehat{\pi}$  in \eqref{pi} defining the foliation ${\cal F}^\omega$  reduces the cluster map $\varphi$ to a  symplectic map $\widehat\varphi$ on $\Rb_+^{2k}$. This means that $\widehat{\pi}\circ\varphi=\widehat{\varphi}\circ\widehat{\pi}$ where $\widehat{\varphi}$ is a symplectic map preserving the symplectic form in \eqref{wnormal}.

\subsection{The symplectic foliation and reduction via a Poisson structure}

We now consider a  Poisson structure $P$ as in \eqref{PP} with $\ker C\neq \{ 0 \}$. The following proposition shows that the symplectic foliation of the regular Poisson manifold $(\Rb_+^N, P)$ is also strictly simple. Its leaves are known as \emph{symplectic leaves}.

\begin{proposition}\label{prop2} Let  $P$ be the Poisson structure on $\Rb_+^N$  defined by the matrix  $C=[c_{ij}]$ with $s$-dimensional kernel as in \eqref{PP}. Then, there exists a surjective submersion $\widetilde\pi: \Rb_+^N \rightarrow \Rb_+^{s}$ 
 with connected fibres $S_\beta= \widetilde\pi^{-1}(\beta)$, which determines  the symplectic foliation ${\cal F}^P$ of $\Rb_+^N$, that is
\[{\cal F}^P=\bigsqcup_{\beta} S_{\beta},\]  
where $S_\beta$ is a symplectic manifold. 

Moreover the leaves $S_\beta$ are algebraic varieties.
\end{proposition}

\begin{proof} Consider again the diffeomorphism $\Phi:  \Rb_+^N \rightarrow  \Rb^N$ given by \eqref{phi}. The pushforward of $P$ by $\Phi$ is the Poisson structure given by
\[P' =\Phi_*P=\sum_{1\leq i<j\leq N} c_{ij} \frac{\partial}{\partial v_i}\wedge\frac{\partial}{\partial v_j},\]
which is a constant Poisson structure on the vector space $\Rb^N$. The symplectic leaves of such Poisson   structure are well known: they are affine subspaces which coincide with the common level sets of a maximal set of independent linear Casimirs. Note that a linear function $f({\bf v})= {\bf u}\cdot {\bf v}$ is a Casimir of $P'$ if and only if ${\bf u} \in \ker C$. 

Therefore the symplectic leaves of $P'$ can be defined as the (connected) fibres of the surjective submersion $\widetilde\Pi:\Rb^N \longrightarrow \Rb^{s} $,
\[\widetilde{\Pi}: \mathbf{v}  \longmapsto (f_1(\mathbf{v}), \ldots, f_{s}(\mathbf{v}))
\]
where $f_1,\ldots ,f_s$ are $s$ independent linear functions of the form
$f_i({\bf v})= {\bf u}_i\cdot {\bf v}$, with ${\bf u}_i \in \ker C$.
It can easily be checked that the fibre $S'$ of $\widetilde\Pi$ through ${\bf v_0}$ can be written as 
\[\mathbf{v_0} + \left ( \ker C\right )^\perp.\] 
In particular, this foliation depends only on $\ker C$ and not on the particular choice of the linear functions $f_1, \ldots , f_{s}$. 

The functions $z_i=\exp(f_i \circ \Phi)$, which have the form 
\[z_i({\bf x})={\bf x}^{{\bf u}_i},\quad \mathbf{u}_i\in \ker C,\]
form a maximal set of independent Casimirs of the Poisson structure $P$. The map $\widetilde\pi: \Rb_+^N  \longrightarrow \Rb_+^s$, 
\begin{equation}
\label{pis}
\widetilde\pi: \mathbf{x}  \longmapsto (z_1(\mathbf{x}), \ldots, z_s(\mathbf{x})),
\end{equation}
is a surjective submersion and each of its fibres, 
\[S_\beta= \left\{\mathbf{x}\in\Rb_+^N:\,\, \widetilde\pi(\mathbf{x})= \beta\right\},\]
is connected since it is $\Phi$-diffeomorphic to a fibre of $\widetilde\Pi$. As each $S_\beta$ is connected and a common level set of $s$ independent Casimirs of $P$, then $S_\beta$ is a symplectic leaf of $P$. 

Finally, as $C=[c_{ij}]$ is an integer matrix, then each vector ${\bf u}$ in $\ker C$ can be chosen to have integer  components, which implies that the fibres $S_\beta$ are algebraic varieties.
\end{proof}

In general, a Poisson structure of the form \eqref{PP} does not allow for reduction of the cluster map $\varphi$ unless $P$ is invariant under $\varphi$, that is, $\varphi_*P=P$ or equivalently    $\varphi$ is a {\it Poisson map}. In such case the submersion $\widetilde\pi$ in \eqref{pis}, defining the foliation ${\cal F}^P$,  satisfies the conditions described in \cite[Theorem 5.1, Lemma 5.2 ]{InEs2}  and so it reduces $\varphi$  to a map $\widetilde\varphi: \Rb^s_+\rightarrow  \Rb^s_+$, that is, $\widetilde\pi \circ \varphi = \widetilde\varphi \circ \widetilde\pi.$

\subsection {Subfoliations and Multiple Reductions}
\label{ssprePois}

We now look for conditions under which the symplectic foliation ${\cal F}^P$ determined by a Poisson structure $P$ is a \emph{subfoliation} of the null foliation ${\cal F}^\omega$ defined by a presymplectic form $\omega$ (or vice-versa). We will also be interested in considering two distinct symplectic foliations, ${\cal F}^{P_1}, {\cal F}^{P_2}$, and on conditions under which ${\cal F}^{P_2}$ is a subfoliation of ${\cal F}^{P_1}$. 
 \begin{definition}\label{defsub}
Let ${\cal F}_1$ and ${\cal F}_2$ be two simple foliations of a manifold $M$ given respectively by surjective submersions $\pi_1:M \rightarrow N_1$ and $\pi_2:M \rightarrow N_2$ with $\dim N_1<\dim N_2<\dim M$. 

The foliation ${\cal F}_2$ is said to be a \emph{simple subfoliation of} ${\cal F}_1$, and will be denoted by ${\cal F}_2\prec {\cal F}_1$, if there exists a surjective submersion $p: N_2\rightarrow N_1$ such that 
\[
p\circ \pi_2 = \pi_1.
\]
\end{definition}
Schematically, a simple subfoliation ${\cal F}_2$ of ${\cal F}_1$ is described by the following diagram

\[\xymatrix{ M   \ar@/_2pc/[rr]_{\pi_1}  \ar @{->} [r]^{\hspace{0cm}\pi_2} & N_2 \ar @{->}[r]^{\hspace{-0.1cm}p}&N_1
 }\]
\begin{remark}
It is clear from Definition~\ref{defsub}, that if ${\cal F}_2\prec {\cal F}_1$ then any leaf of ${\cal F}_2$ is contained in a leaf of ${\cal F}_1$, that is, ${\cal F}_2$ is indeed a subfoliation of ${\cal F}_1$. 
\end{remark}
\begin{proposition}
\label{propnest}
Let $\omega$ be the presymplectic form in \eqref{omega} determined by the matrix $B$.  Let $P_1$ and $P_2$ be Poisson structures on $\Rb_+^N$ determined by $C_1$ and $C_2$ as in \eqref{PP}. Then:
\begin{enumerate}
\item ${\cal F}^{P_i}\prec {\cal F}^\omega$ if and only if  $\im B \subset \ker C_i$.
\item ${\cal F}^\omega \prec{\cal F}^{P_i}$ if and only if $\ker C_i \subset \im B$.
\item ${\cal F}^{P_2}\prec {\cal F}^{P_1}$ if and only if $\ker C_1 \subset \ker C_2$.
\end{enumerate} 
\end{proposition}
\begin{proof} We only prove the first statement since the proofs of the remaining statements are analogous.
 
The submersions $\widehat\pi $ and $\widetilde \pi_{i}$ whose fibres define (respectively) the foliations ${\cal F}^\omega$ and  ${\cal F}^{P_i}$ were obtained in the proofs of propositions~\ref{prop1} and \ref{prop2} by means of the diffeomorphism $\Phi$ given by \eqref{phi}. 
Recall from those proofs that the leaf $N$ of ${\cal F}^\omega$ through $\mathbf{x}_0\in\Rb_+^N$ is $\Phi$-diffeomorphic to the affine subspace of $\Rb^N$
\[N' = \Phi({\bf x}_0) + \left ( \im B\right )^\perp\]
whereas the  leaf $S^i$ of ${\cal F}^{P_i}$ through the same point ${\bf x}_0$ is $\Phi$-diffeomorphic to
\[{S^i}'= \Phi({\bf x}_0) + \left ( \ker C_i\right )^\perp.\]
If  ${\cal F}^{P_i}\prec{\cal F}^\omega$ then $S^i\subset N$ which is equivalent to  ${S^i}'\subset N'$ and consequently $\im B \subset \ker C_i$. 

Conversely, if $\im B \subset \ker C_i$ we can choose the submersions $\widehat\pi :\Rb_+^N\rightarrow \Rb_+^{2k}$ from \eqref{pi} and $\widetilde \pi_{i}:\Rb_+^N\rightarrow \Rb_+^{s_i}$ from \eqref{pis} in such a way that 
\[
\widetilde \pi_{i} ({\bf x}) = \left ( \widehat\pi ({\bf x}), z_{2k+1}(\mathbf{x}), \ldots, z_{s_i}(\mathbf{x})\right ).
\]
Thus it is clear that $\widehat\pi  = p_i\circ\widetilde \pi_{i}$, where $p_i: \Rb_+^{s_i}\rightarrow \Rb_+^{2k}$ is the projection onto the first $2k$ coordinates, showing that ${\cal F}^{P_i}\prec{\cal F}^\omega$. 
\end{proof}

As an immediate consequence of the last proposition and of the characterization of the null foliation ${\cal F}^\omega$ given in Proposition~\ref{prop1} we have the following corollary.
\begin{corollary}
Under the conditions of Proposition~\ref{propnest}, if ${\cal F}^{P}\prec {\cal F}^\omega$  then $\omega\restr{S}\equiv 0$ for any symplectic  leaf $S$ of ${\cal F}^{P}$, that is, each symplectic leaf $S$ of $P$ is \emph{isotropic} with respect to $\omega$.
\end{corollary}
We conclude this section by remarking that the condition $\im B\subset \ker C$ in the first statement of Proposition~\ref{propnest} is equivalent to $CB=[ 0 ]$ and imposes several restrictions on the matrix $C$. For instance, just by dimension counting, it is clear that if $B$ has rank equal to $N-1$ then the only possible matrix $C$ satisfying this condition is the zero matrix. 

On the other hand, for some matrices $B$, like the one in the following  example, there exist nontrivial matrices $C_1$ and $C_2$ with different ranks, satisfying $C_iB=[ 0 ]$. This case will be explored in the last section, where two distinct Poisson structures are used to gain more insight into the dynamics of a cluster map. 

Finally we note that the condition $CB=[ 0 ]$ appears in the context of cluster algebra theory (\cite{GeShVa} and \cite{InNa}) to characterize a Poisson structure (determined by $C$) as being \emph{compatible} with the cluster algebra ${\cal A}(B)$ when $B$ is singular. For our purpose of reducing a cluster map via a Poisson structure, there is no reason \emph{a priori} to consider only Poisson structures which are compatible with ${\cal A}(B)$. However, the proposition above together with the results in \cite{InNa} show that the condition of compatibility with ${\cal A}(B)$ appears naturally when the symplectic foliation of the Poisson structure is a subfoliation of ${\cal F}^\omega$.

\begin{example} \label{ex3} 

Consider the following skew-symmetric matrix $B$ of rank 2, which defines a 1-periodic quiver (see \cite{FoMa} for the classification of these quivers):
 \small\[B=\begin{bmatrix}
0&1&0&-1&-1&0&1\\
-1&0&1&1&0&-1&0\\
0&-1&0&1&1&0&-1\\
1&-1&-1&0&1&1&-1\\
1&0&-1&-1&0&1&0\\
0&1&0&-1&-1&0&1\\
-1&0&1&1&0&-1&0
\end{bmatrix}.
\]
\normalsize
The following skew-symmetric matrices $C_1$ and $C_2$ satisfy the condition $C_iB=[0]$ (only the upper triangular part of the matrices is shown):
\small
\begin{equation}\label{C7nos}
C_1=\begin{bmatrix}
0&1&1&2&3&3&4\\
&0&1&1&2&3&3\\
&&0&1&1&2&3\\
&&&0&1&1&2\\
&&&&0&1&1\\
&&&&&0&1\\
&&&&&&0\\
\end{bmatrix},
\quad C_2=\begin{bmatrix}
0&1&-1&0&1&-1&0\\
&0&1&-1&0&1&-1\\
&&0&1&-1&0&1\\
&&&0&1&-1&0\\
&&&&0&1&-1\\
&&&&&0&1\\
&&&&&&0\\
\end{bmatrix}.\end{equation}
\normalsize
Moreover, it can be checked that $\rk C_1=4$, $\rk C_2=2$ and $\ker C_1\subset \ker C_2$. By Proposition~\ref{propnest} we have 
\[{\cal F}^{P_2}\prec {\cal F}^{P_1} \prec {\cal F}^\omega,\]
where the symplectic leaves of ${\cal F}^{P_2}$ are 2-dimensional, the symplectic leaves of ${\cal F}^{P_1}$ are 4-dimensional  and the null leaves of ${\cal F}^\omega$ are 5-dimensional. 

Submersions $\widetilde{\pi}_2, \widetilde{\pi}_1$  and $\widehat{\pi}$ defining  the foliations ${\cal F}^{P_2}, {\cal F}^{P_1}$ and  ${\cal F}^{\omega}$, respectively, are given by:
\[
\widetilde\pi_{2}({\bf x})=(y_1,y_2,y_3, y_4,y_5), \quad \widetilde\pi_{1}({\bf x})=(y_1,y_2,y_3), \quad \widehat\pi ({\bf x})=(y_1,y_2),
\]
with ${\bf x}= (x_1,x_2,\ldots ,x_7)$ and 
\[y_1=\frac{x_1x_6}{x_3x_4},\,\,y_2=\frac{x_2x_7}{x_4x_5}, \,\,y_3=\frac{x_1x_7}{x_4^2}, \,\, y_4= x_1x_2x_3, \,\,  y_5=x_2x_3x_4.
\]
\end{example}

\section{Multiple reductions and the dynamics of cluster maps}\label{din}

In this section we will focus on the consequences of reduction, via presymplectic or via Poisson structures, to the dynamics of a cluster map $\varphi$. The next proposition describes the dynamical behaviour of a map on the leaves of the foliation associated   to a reduction of it.

Recall that a \emph{first integral} of a map $f$ is a non constant real-valued function $I$ such that $I\circ f=I$. Also, a map $f$ is said to be \emph{globally $p$-periodic} if its   $p${th}-iterate is the identity, that is,  $f^{(p)}=Id$.

\begin{proposition}
\label{reduc}
Let ${\cal F}$ be a strictly simple foliation of $M$ given by  $\pi:M\rightarrow N$ and suppose that $(g, \pi)$ is a reduced system of $f: M \rightarrow M$. Then 
\begin{enumerate}
\item the map $f$ sends leaves of ${\cal F}$ to leaves of ${\cal F}$, more precisely
\[f (L_\alpha) \subset L_{g(\alpha)};\]
\item the leaf $L_\alpha$ of ${\cal F}$ is invariant under the map $f^{(p)}$ if and only if $\alpha$ is a $p$-periodic point of $g$, that is, $g^{(p)}(\alpha) = \alpha$;
\item $g$ is a globally $p$-periodic map   if and only if $\pi$ is a vector-valued first integral of $f^{(p)}$, that is, $\pi\circ f^{(p)}=\pi$. 
\end{enumerate}
\end{proposition}

\begin{proof}
All the statements follow from the  definition of a reduced system, in particular from the identity 
\begin{equation}
\label{fgpi} 
\pi \circ f = g \circ \pi,
\end{equation} 
and from the fact that any leaf $L_\alpha$ of the foliation is given by
\[L_\alpha = \{ {\bf x} \in M: \pi ({\bf x}) = \alpha\}.\]
For the first statement assume that ${\bf x}$ belongs to $L_\alpha$. Then \eqref{fgpi} shows that $\pi (f({\bf x})) = g (\alpha)$, that is, $f({\bf x})$ belongs to $L_{g(\alpha)}$. 

Concerning the second and third statements, note that \eqref{fgpi} implies 
\[\pi \circ f^{(n)} = g^{(n)} \circ \pi, \quad n=1,2, \ldots\]
and the results follow directly from this identity. 
\end{proof}

\begin{proposition} 
\label{redprespois}
Let $(\widehat{\varphi}, \widehat{\pi})$ and $(\widetilde{\varphi}, \widetilde{\pi})$ be reduced systems of the cluster map $\varphi$, arising from reduction  via the presymplectic form $\omega$ in \eqref{omega}   and via a $\varphi$-invariant Poisson structure $P$, respectively.  Let   ${\cal F}^\omega$ be the null foliation defined by $\widehat{\pi}$ and ${\cal F}^P$ the symplectic foliation defined by $\widetilde{\pi}$. Denote the null leaves by $N_\alpha$ and  the symplectic leaves by  $S_\beta$. Then,
\begin{enumerate}
\item the cluster map $\varphi$ sends each leaf $N_\alpha$ (resp. $S_\beta$) of  ${\cal F}^\omega$ (resp.  of ${\cal F}^P$) to the leaf $N_{\widehat{\varphi}(\alpha)}$ (resp. $S_{\widetilde{\varphi}(\beta)}$) of the same foliation;
\item a leaf $N_\alpha$ (resp. $S_\beta$) is invariant under $\varphi^{(p)}$ if and only if $\alpha$ (resp. $\beta$) is a $p$-periodic point of $\widehat{\varphi}$ (resp. $\widetilde{\varphi}$);
\item if $\beta$ is a $p$-periodic point of $\widetilde{\varphi}$, the restriction of $\varphi^{(p)}$ to $S_\beta$ is a symplectic map.
\end{enumerate}
\end{proposition}

\begin{proof}
The first two statements are  just a rephrasing of items {\it 1} and {\it 2} in Proposition~\ref{reduc}. 
The third statement is a more relevant property of reduction via Poisson structures. Its proof follows from classical theory on Poisson and symplectic manifolds (see for example \cite{LiMa} and \cite{AWein}) as we now describe. 

The symplectic structure on $S_\beta$ is given by the nondegenerate Poisson structure induced by the Poisson structure $P$ on $\Rb_+^N$. This means that the inclusion $i:S_\beta \hookrightarrow \Rb_+^N$ is a Poisson map. 

The fact that $\varphi$ is a Poisson map implies that $\varphi^{(p)}$ is also a Poisson map. The composition $\varphi^{(p)}\circ i$, which is precisely the restriction $\varphi^{(p)}\restr{S_\beta}$, is therefore a Poisson map. Since the Poisson structure on $S_\beta$ is nondegenerate, then any map preserving this Poisson structure must preserve the associated symplectic structure. In other words $\varphi^{(p)}\restr{S_\beta}$ is symplectic.
\end{proof}

We end this section by explaining  how multiple reductions of the same cluster map  may simplify the study of the dynamics of the unreduced map. 
For instance, suppose we have two reduced systems $(g_1, \pi_1)$ and $(g_2, \pi_2)$ of the same map $f:M\rightarrow M$, such that the associated foliations  ${\cal F}_1$ and ${\cal F}_2$ are strictly simple foliations and satisfy ${\cal F}_2\prec {\cal F}_1$.  In this case, the study of the  dynamics of $f$ is effectively simplified,  in the sense that we can study (sequentially)  the dynamics of the reduced map $g_2$ by using the map $g_1$ and then study the dynamics of $f$ using the map $g_2$. 
This procedure is made precise in the next proposition and the examples of the following section illustrate it.

\begin{proposition}
\label{redpi}
Let ${\cal F}_1$ and ${\cal F}_2$ be strictly simple foliations of $M$  given by the surjective submersions $\pi_1: M\rightarrow N_1$ and $\pi_2: M\rightarrow N_2$, respectively. Suppose that ${\cal F}_2\prec {\cal F}_1$ and let $p: N_2\rightarrow N_1$ be a submersion in the conditions of Definition~\ref{defsub}. 

If $(g_1, \pi_1)$ and $(g_2, \pi_2)$ are reduced systems of $f:M\rightarrow M$, then $(g_1, p)$ is a reduced system of $g_2$. 
\end{proposition}

\begin{proof}
Because $(g_1, \pi_1)$ is a reduced system of $f$, we have $\pi_1\circ f=g_1\circ \pi_1$. As ${\cal F}_2\prec {\cal F}_1$, we also have $\pi_1=p\circ\pi_2$ which leads to
\[p\circ \pi_2\circ f=g_1\circ p\circ \pi_2.\]
As $(g_2, \pi_2)$ is also a reduced system of $f$, we obtain
\[p\circ g_2 \circ \pi_2=g_1\circ p\circ \pi_2\]
and the conclusion follows from surjectivity of $\pi_2$. 
\end{proof}

The procedure described in the last proposition of assigning  a third reduced system to two existing ones can be used with great advantage to study the original map $f$ when we have  several foliations (associated to reduced systems) verifying 
\[
{\cal F}_m\prec  \cdots \prec {\cal F}_2\prec{\cal F}_1.\]
Such a set of simple foliations will be called a \emph{flag of simple foliations}. 
 An example of this type, where we have two symplectic foliations and a null foliation satisfying
${\cal F}^{P_2}\prec {\cal F}^{P_1} \prec{\cal F}^\omega,$
will be comprehensively treated in the next section. 

\section{Examples}

In this section we apply the results obtained in the previous sections to draw conclusions on the dynamics of two cluster maps. 
Our first example is the Somos-5 cluster map which illustrates the simplest situation of applicability  of our results: existence of a flag ${\cal F}^{P}\prec {\cal F}^{\omega}$ and of a fixed point of the reduced map $\widehat{\varphi}$.
The second example is a cluster map in dimension 7, also associated to a 1-periodic quiver,  which falls into a more specific case: existence of a flag ${\cal F}^{P_2}\prec {\cal F}^{P_1} \prec {\cal F}^\omega$ and global periodicity of the reduced map $\widehat{\varphi}$. 

Although the examples chosen are maps associated to 1-periodic quivers, the results proved in the previous sections apply to higher periodic quivers.

\begin{example} {\bf Somos-5}

It is well known that the Somos-5 recurrence
\begin{equation}
\label{somos5rep}
x_{n+5}x_n= x_{n+1} x_{n+4}+x_{n+2}x_{n+3}
\end{equation}
is a  recurrence arising from the mutation-periodic quiver $Q_B$ with $B$ as in Example~\ref{ex1} with $r=s=1$. 
 The associated cluster map can be reduced using the presymplectic form \eqref{omega} and the reduced map $\widehat{\varphi}$ is a QRT map of the plane \cite{QRT} which is known to be integrable in the Liouville sense (see for instance \cite{Duist} and \cite{QRT2} for the integrability of QRT maps). In fact, a first integral of this QRT map defines an elliptic fibration  of the plane whose generic fibres are curves of genus 1. The general solution of the Somos-5 recurrence corresponding to elliptic curves of genus 1 was obtained in \cite{Ho} in  terms of sigma functions. Using the approach developed  in the previous sections  we obtain (Proposition~\ref{propS} below) the general solution of the Somos-5 recurrence corresponding to the singular case (genus zero curves).
 
The cluster map associated to the Somos-5 recurrence is 
\begin{equation}\label{somos511}
\varphi(x_1, \ldots, x_5)= \left(x_2,x_3,x_4,x_5, \frac{x_2 x_5+x_3x_4}{x_1}\right).
\end{equation}
The Poisson structure $P$ on $\Rb_+^5$ given by the skew-symmetric matrix $C=[c_ij]$ with $c_ij=j-i$,
is known to be invariant under $\varphi$. This matrix $C$ has a $3$-dimensional kernel and  $\im B\subset \ker C$.  By Proposition~\ref{propnest}, ${\cal F}^P$ is a simple subfoliation of  ${\cal F}^\omega$. 
 Submersions defining these foliations produce the reduced systems $(\widetilde{\varphi}, \widetilde{\pi})$ and $(\widehat{\varphi}, \widehat{\pi})$, where 
\begin{align*}
\widetilde{\varphi }:  {\bf x} \longmapsto \left(y_2, \frac{1+y_2}{y_1 y_2}, \frac{1+y_2}{y_1 y_2 y_3}\right), & \qquad \widetilde\pi :  {\bf x} \longmapsto (y_1,y_2, y_3) \\
\widehat{\varphi} :  {\bf x} \longmapsto \left(y_2, \frac{1+y_2}{y_1 y_2}\right), & \qquad \widehat\pi :  {\bf x} \longmapsto (y_1,y_2)
\end{align*}
with $\mathbf{x}=(x_1, \ldots,  x_5)\in\Rb_+^5$ and
\begin{equation}
\label{y1y2}
y_1=\frac{x_1 x_4}{x_2 x_3},\quad  y_2=\frac{x_2 x_5}{x_3 x_4}, \quad y_3=\frac{x_3 x_5}{x_4^2}.
\end{equation}
For more details on these reductions see  \cite{FoHo1} and \cite{InEs2}.

By Proposition~\ref{redpi}, $(\widehat{\varphi}, p)$ is a reduced system of $\widetilde{\varphi}$ with 
$p (y_1,y_2,y_3)=(y_1,y_2)$. Schematically: 

\[\xymatrix{ \Rb_+ ^5  \ar@/^2pc/[rr]^{\widehat\pi } \ar @{->} [r]^{\hspace{-0.1cm}\widetilde \pi}\ar @{->} [d]_{\varphi} & \Rb^3_+ \ar @{->}[d]^{\widetilde\varphi}\ar @{->}[r]^{\hspace{0.0cm}p} &\Rb_+ ^2\ar @{->}[d]^{\widehat\varphi}\\
 \Rb_+ ^{5}   \ar@/_2pc/[rr]_{\widehat\pi }  \ar @{->} [r]_{\hspace{-0.1cm}\widetilde \pi} & \Rb^3 \ar @{->}[r]_{\hspace{0.0cm}p}&\Rb_+ ^2
 }\]

It is easy to see that  the map $\widehat\varphi$ has a unique fixed   point, the point $(r,r)$  
with $r$ the real (positive) root of $x^3=1+x$, and  no points of minimal period 2. A direct application of Proposition~\ref{reduc} to the reduced system  $(\widehat\varphi,p)$ of $\widetilde\varphi$ shows that the leaf
$L= \{ {\bf y} \in \Rb_+^3: y_1=r, y_2=r\}$
is invariant under $\widetilde\varphi$. Let $h$ denote the restriction  of $\widetilde\varphi$ to $L$. Using the natural $y_3$ coordinate on $L$,  the expression of $h$ is
\[h(y_3)=\frac{r}{y_3},\]
which is a globally 2-periodic map with a unique fixed point: $y_3=\sqrt{r}$. 

Summing up, the point $(r,r,\sqrt{r})$ is the unique fixed point of the map $\widetilde\varphi$ and its 2-periodic points are precisely the points on $\Rb_+^3$ of the form $(r,r,\lambda)$. Applying  Proposition~\ref{redprespois} to the reduced system $(\widetilde\varphi,\widetilde\pi)$ of $\varphi$,  the following conclusions can be withdrawn for the symplectic leaves of ${\cal F}^P$:
\begin{itemize}
\item
$S_{(r,r,\sqrt{r})} = \{ {\bf x}\in \Rb_+^5: x_1x_4=r x_2 x_3, \, x_2x_5=r x_3 x_4, x_3x_5=\sqrt{r} x_4^2\}$
is invariant under $\varphi$ and the restriction of $\varphi$ to $S_{(r,r,\sqrt{r})}$ is symplectic; 
\item
$S_{(r,r,\lambda)} = \{{\bf x}\in \Rb_+^5: x_1x_4=r x_2 x_3, \, x_2x_5=r x_3 x_4,\, x_3x_5=\lambda x_4^2\}$,  for $\lambda\in\Rb_+$, 
is invariant under $\varphi^{(2)}$ and the restriction of $\varphi^{(2)}$ to $S_{(r,r,\lambda)}$ is symplectic.
\end{itemize}

Choosing natural coordinates $(x_3,x_4)$ on each symplectic leaf, the computation of the restricted maps 
$h_1=\varphi\restr{S_{(r,r,\sqrt{r})}}$ and $h_2=\varphi^{(2)}\restr{S_{(r,r,\lambda)}}$ gives
\[h_1(x_3,x_4)=\left (x_4,\sqrt{r} \frac{x_4^2}{x_3}\right ) \quad \mbox{and} \quad h_2(x_3,x_4)= \lambda \left ( \frac{x_4^2}{x_3},\frac{r \, x_4^3}{x_3^2}\right ).\]

The expression of the iterates of $h_1$ and $h_2$ can be obtained directly  from Lemma 1 of \cite{InHeEs}. These are given, for $n\geq 1$, by
\[h_1^{(n)} =r^{n(n-1)/4}\left (  \frac{x_4^n}{x_3^{n-1}}, r^{n/2}\frac{x_4^{n+1}}{x_3^n}\right ),\quad h_2^{(n)} =\lambda^n r^{n(n-1)}\left (  \frac{x_4^{2n}}{x_3^{2n-1}}, r^{n}\frac{x_4^{2n+1}}{x_3^{2n}}\right )\]
Hence, we are led to the following proposition.
\begin{proposition}
\label{propS}
Consider the Somos-5 recurrence in \eqref{somos5rep}
and let $r$ be the real root of $x^3=1+x$.  Let $(x_1,\ldots,x_5)$ be initial data satisfying the identities
\[x_1x_4=r x_2 x_3, \,\,\, x_2x_5=r x_3 x_4,\,\, x_3x_5= \lambda x_4^2\]
with $\lambda\in\Rb_+$. Then,
\begin{enumerate}
\item if $\lambda=\sqrt{r}$, the  corresponding solution of the Somos-5 recurrence is given by
\[ x_{n+5}=r^{(n+2)(n+1)/4}\,\,\frac{x_4^{n+2}}{x_3^{n+1}}, \quad n\geq 1;\]
\item if $\lambda\neq\sqrt{r}$,
 the   corresponding  solution of the Somos-5 recurrence is given by
\[x_{2n+4}=\lambda^{n}r^{n^2}\,\,\frac{x_4^{2n+1}}{x_3^{2n}} , \quad x_{2n+5}=\lambda^{n+1} r^{n(n+1)}\,\, \frac{x_4^{2n+2}}{x_3^{2n+1}}, \quad n\geq 1.\]
\end{enumerate}
\end{proposition}
\end{example}

\begin{example} {\bf Dynamics of a cluster map in dimension 7}
\label{ex4}

The most general form of the Somos-7 recurrence is
\begin{equation}\label{somos7}
x_{n+7}x_n=\alpha x_{n+1} x_{n+6} +\beta x_{n+2}x_{n+5}+\gamma x_{n+3} x_{n+4}.
\end{equation}
 Clearly this  recurrence does not fit into the setting of cluster maps unless  one and just one of the parameters $\alpha, \beta$ or $\gamma$ is zero.   In that case  the respective recurrence can be included in the cluster algebra framework by adding two extra (frozen) nodes to a (coefficient-free) 1-periodic quiver of 7 nodes represented by a certain matrix $B$, see  \cite[Section 10]{FoMa}. 
This matrix  has rank equal   to 4 when  $\alpha=0$ or $\gamma=0$ and rank 2 when $\beta=0$. This means that the respective reduced map $\hat{\varphi}$ is a 4-dimensional map when $\alpha=0$ or $\gamma=0$ and a 2-dimensional map when $\beta=0$. 

Three first integrals for the Somos-7   recurrence were given in \cite[Theorem 6.4]{FoHo2} where it was proved that they descend to first integrals of the 4D and 2D (symplectic) reduced maps. Namely, it was proved that both reduced maps are Liouville integrable: there are two first integrals of $\widehat{\varphi}$ in involution  if $\alpha=0$ or $\gamma=0$ and one first  integral if $\beta=0$.

Here we consider the  recurrence \eqref{somos7} with $\beta=0$ and $\alpha=\gamma=1$.  The map defining this recurrence is the map 
\begin{equation}\label{csomos7}
\varphi(x_1, \ldots,x_7)=\left(x_2,x_3, \ldots, x_7, \frac{x_2x_7+x_4x_5}{x_1}\right),
\end{equation}
which is the cluster map associated to the matrix $B$ in Example~\ref{ex3}.

As seen in that example,  there exists a flag of strictly simple subfoliations ${\cal F}^{P_2}\prec {\cal F}^{P_1}\prec  {\cal F}^\omega$ 
where $\omega$ is the presymplectic form defined by the matrix $B$, and $P_1, P_2$ are the Poisson structures defined respectively by the matrices $C_1$ and $C_2$  in \eqref{C7nos}.  Submersions $\widehat\pi$, $\widetilde\pi_1$ and $\widetilde\pi_2$ defining these foliations were also obtained in Example~\ref{ex3}. 

Straigthforward computations show that $\varphi$ is a Poisson map with respect to both $P_1$ and  $P_2$, and so there are three reduced systems $(\widehat{\varphi}, \widehat{\pi})$, $(\widetilde{\varphi}_1, \widetilde{\pi}_1)$ and $(\widetilde{\varphi}_2, \widetilde{\pi}_2)$ of $\varphi$. The respective reduced maps can be easily computed and have the form

\begin{align}\label{psi16}
\widehat{\varphi}(y_1,y_2)=\left(y_2, \frac{1+y_2}{y_1}\right),\quad \widetilde{\varphi}_1(y_1,y_2,y_3)&=\left(\widehat{\varphi}(y_1,y_2); \frac{y_2 (1+y_2)}{y_3}\right),\\ \label{psi17}
\widetilde{\varphi}_2(y_1, \ldots, y_5)=\left( \widetilde\varphi_1(y_1,y_2,y_3); y_5, \frac{y_3y_5^2}{y_2y_4} \right).
\end{align}

\noindent This can be summarised in the following commutative diagram:
\begin{equation}
\label{commd}
\xymatrix{ \Rb^7_+\ar@/^3pc/[rrr]^{\widehat\pi }\ar@/^2pc/[rr]_{\widetilde\pi_1}    \ar[r]^{\widetilde\pi_2} \ar[d]_{\varphi} & \Rb^5_+ \ar[d]^{\widetilde{\varphi}_2}\ar [r]^{{p_2}}& \Rb^3_+ \ar[d]^{\widetilde{\varphi}_1}\ar [r]^{{p}_1} &\Rb ^2_+\ar[d]^{\widehat\varphi}\\
 \Rb^7 _+   \ar@/_3pc/[rrr]_{\widehat\pi } \ar@/_2pc/[rr]^{\widetilde\pi_1}   \ar [r]_{\widetilde\pi_2} & \Rb^5_+\ar [r]_{{p_2}}& \Rb^3_+\ar [r]_{{p}_1}&\Rb^2_+
 }
 \end{equation}
\noindent where $p_2:\Rb_+^5\rightarrow \Rb_+^3$ and ${p}_1: \Rb_+^3\rightarrow \Rb_+^2$  are the canonical projections 
\begin{equation}\label{proj7}
{p_2}(y_1,\ldots, y_5)=(y_1,y_2,y_3), \quad p_1(y_1,y_2, y_3)=(y_1,y_2).
\end{equation}

In order to describe the dynamics of $\varphi$ we study sequentially  three reduced systems (provided by Proposition~\ref{redpi}) starting with the lowest dimensional reduced map $\widehat\varphi$. 

The map $\widehat{\varphi}$ in \eqref{psi16}  is the well-known Lyness map, which is also a QRT map and therefore an integrable map  \cite{FoHo2}. Moreover,
 $\widehat{\varphi}$ is a globally 5-periodic map such that all the points in $\Rb^2_+$ have minimal period 5 apart from the point $(\phi,\phi)$, with $\phi=\frac{1+\sqrt{5}}{2}$ the golden number, which is a fixed point.

As ${\cal F}^{P_1}\prec {\cal F}^\omega$, Proposition~\ref{redpi} implies that $(\widehat\varphi,p_1)$ is a reduced system of $\widetilde{\varphi}_1$. 
 The application of Proposition~\ref{reduc} to $(\widehat\varphi,p_1)$  leads to $\widetilde\varphi_1$-invariance properties of the leaves
\begin{equation}
\label{lalfa}
L_{(a,b)} = \{ {\bf y} \in \Rb_+^3: y_1= a, y_2=b\}
\end{equation}
 of the 1-dimensional foliation ${\cal F}$ of $\Rb_+^3$. These properties are as follows:
\begin{enumerate}
\item $L_{(\phi,\phi)}$ is invariant under $\widetilde\varphi_1$;
\item for any $(a,b)\neq (\phi,\phi)$, $L_{(a,b)}$ is invariant under $\widetilde\varphi_1^{(5)}$ and not invariant under $\widetilde\varphi_1^{(n)}$ for $1\leq n<5$. This means,  by Proposition~\ref{reduc}, that the $\widetilde\varphi_1$-orbit of any point in $L_{(a,b)}$ circulates between five distinct leaves of ${\cal F}$:
\[L_{(a,b)} \rightarrow L_{\widehat{\varphi}(a,b)} \rightarrow L_{\widehat{\varphi}^{(2)}(a,b)} \rightarrow L_{\widehat{\varphi}^{(3)}(a,b)}\rightarrow L_{\widehat{\varphi}^{(4)}(a,b)}.\]
\end{enumerate}
Studying the restrictions of $\widetilde\varphi_1$ to $L_{(\phi,\phi)}$ and of $\widetilde\varphi_1^{(5)}$ to $L_{(a,b)}$ we  obtain the full description of the dynamics of $\widetilde\varphi_1$ in the following lemma.

\begin{lemma}\label{Poisson1}
Let  $\widetilde{\varphi}_1:\Rb^3_+\rightarrow \Rb^3_+$ be the map in \eqref{psi16} and ${\cal F}$
 the 1-dimensional foliation  of $\Rb_+^3$ whose leaves $L_{(a,b)}$ are given by \eqref{lalfa}. Then, with $\phi=\frac{1+\sqrt{5}}{2}$ 
\begin{enumerate}
\item $L_{(\phi,\phi)}$ contains the unique fixed point of $\widetilde{\varphi}_1$, the point $P=(\phi,\phi, \sqrt{\phi^3})$, and all the other points in $L_{(\phi,\phi)}$ are periodic points of $\widetilde{\varphi}_1$ with minimal period 2;
\item  $L_{(a,b)}$, with $(a,b)\neq (\phi,\phi)$,  contains precisely one periodic point of $\widetilde{\varphi}_1$ with minimal period 5, the point $(a,b, \sqrt{g(a,b)})$ with 
\begin{equation}
\label{gab}
g(a,b)= \frac{ab(a+1)(b+1)}{(a+b+1)}.
\end{equation}
Any other point in $L_{(a,b)}$ is a periodic point of $\widetilde{\varphi}_1$ with minimal period 10.
\end{enumerate}
In particular, $\widetilde{\varphi}_1$ is globally 10-periodic.
\end{lemma}

\begin{proof}
The restriction $h_1$ of $\widetilde{\varphi}_1$ to $L_{(\phi,\phi)}$ is given 
$h_1(y_3)=\frac{\phi^3}{y_3},$
which is a globally 2-periodic map with exactly one fixed point: $\sqrt{\phi^3}$.

We immediately conclude that $\left(\phi,\phi,\sqrt{\phi^3}\right)$ is the only  fixed point of $\widetilde{\varphi}_1$ and all the other points of $L_{(\phi,\phi)}$ are periodic points of $\widetilde{\varphi}_1$ with minimal period 2. 

For each $(a,b)\in\Rb^2_+\backslash\{{(\phi,\phi)}\}$ the restriction of $\widetilde{\varphi}_1^{(5)}$ to $L_{(a,b)}$ is given by 
\[
h_5(y_3)=\frac{g(a,b)}{y_3},\]
with $g$ as in \eqref{gab}.
The map $h_5$ is globally 2-periodic and has exactly one fixed point:  $\sqrt{g(a,b)}$. 

Consequently, the restriction of $\widetilde{\varphi}_1$ to $L_{(a,b)}$ has precisely one point of minimal period 5, the point $(a,b,\sqrt{g(a,b)})$, and all the remaining points of $L_{(a,b)}$ are periodic points of $\widetilde{\varphi}_1$ with minimal period 10.
\end{proof}

Applying again Proposition~\ref{redpi}, we conclude that $(\widetilde{\varphi}_1,p_2)$ is a reduced system of $\widetilde{\varphi}_2$ (see  also the commutative diagram \eqref{commd}), and so the information from Lemma~\ref{Poisson1} concerning the dynamics of $\widetilde{\varphi}_1$ allows us to draw conclusions on the dynamics of $\widetilde{\varphi}_2$. In fact, 
combining the previous lemma with Proposition~\ref{reduc} we are led to the $\widetilde\varphi_2$-invariance properties of the leaves 
\begin{equation}
\label{l'abc}
L'_{(a,b,c)} = \{ {\bf y} \in \Rb_+^5: y_1= a, y_2=b, y_3=c\}
\end{equation}
of the 2-dimensional foliation ${\cal F}'$ of $\Rb_+^5$ defined by $p_2$. More precisely:
\begin{enumerate}
\item $L'_{P}$, with $P=(\phi,\phi, \sqrt{\phi^3})$, is invariant under $\widetilde\varphi_2$;
\item if  $Q=(\phi,\phi, c)\neq P$, then the $\widetilde\varphi_2$-orbit of any point in $L'_{Q}$ circulates between the leaves $L'_{Q}$ and $L'_{\widetilde{\varphi}_1(Q)}$ of ${\cal F}'$;
\item if $Q=(a,b,\sqrt{g(a,b)})$ with $(a,b)\neq (\phi,\phi)$, then  the $\widetilde\varphi_2$-orbit of any point in $L'_{Q}$ circulates between five distinct leaves of ${\cal F}'$:
\[L'_{Q}\rightarrow L'_{\widetilde{\varphi}_1(Q)}\rightarrow L'_{\widetilde{\varphi}_1^{(2)}(Q)}\rightarrow  L'_{\widetilde{\varphi}^{(3)}_1(Q)}\rightarrow L'_{\widetilde{\varphi}_1^{(4)}(Q)};\]
\item in all other cases the $\widetilde\varphi_2$-orbit of any point in $L'_{Q}$ circulates between ten distinct leaves of ${\cal F}'$:
$L'_{Q}\rightarrow L'_{\widetilde{\varphi}_1(Q)}\rightarrow \ldots \rightarrow L'_{\widetilde{\varphi}_1^{(9)}(Q)}$.
\end{enumerate}
Again, studying the restrictions of $\widetilde\varphi_2$ to $L'_{P}$, and of $\widetilde\varphi_2^{(2)}, \widetilde\varphi_2^{(5)}$ or $\widetilde\varphi_2^{(10)}$ to other leaves $L'_Q$, we will conclude that the reduced map $\widetilde\varphi_2$  has no periodic points. In spite of the apparent simplicity of the result, some intermediate results and nontrivial computations are needed. 

\begin{lemma}\label{Poisson2}
The map $\widetilde{\varphi}_2:\Rb^5_+\rightarrow \Rb^5_+$  given in \eqref{psi17} has no periodic points.
\end{lemma}

\begin{proof} From the above description of the invariance of the leaves  \eqref{l'abc} of ${\cal F}'$, it is enough to study the restrictions of $\widetilde\varphi_2$, $\widetilde\varphi_2^{(2)}$, $\widetilde\varphi_2^{(5)}$ and $\widetilde\varphi_2^{(10)}$ to the appropriate leaves of ${\cal F}'$. All these restricted maps are symplectic birational maps of the plane (see \cite{Blanc}) having the particular form
\begin{equation}
\label{gamma}
f(x,y)=\left(k x^m y^n, l x^p y^q\right),
\end{equation}
with $k,l$ real positive constants and $m,n,p,q$ integers satisfying $mq-np=~1$, which were comprehensively studied in \cite{InHeEs2}. 

The restriction $h_1$ of $\widetilde{\varphi}_2$ to $L'_P$ is given in the natural coordinates $(y_4,y_5)$ by 
\[h_1(y_4,y_5)=\left(y_5,\sqrt{\phi}\frac{y_5^2}{y_4}\right).\] 
By Lemma 2  in \cite{InHeEs2} this map has no periodic points and all its components go to infinity. Consequently $\widetilde{\varphi}_2$ has no periodic points in $L'_P$.

The restriction of $\widetilde{\varphi}_2^{(2)}$ to $L'_Q$, where $Q=(\phi,\phi, c)\neq P$, is given by 
\[
h_2(y_4,y_5)=c\left(\frac{1}{\phi}\,\frac{y_5^2}{y_4},\frac{y_5^3}{y_4^2}\right),
\]
and the restriction of $\widetilde{\varphi}_2^{(5)}$ to $L'_Q$, where $Q=(a,b,\sqrt{g(a,b)})$ and  $(a,b)\neq (\phi,\phi)$, is given by 

 \[
h_5(y_4,y_5)=\frac{(1+b)g(a,b)}{b}\left(a\,\frac{y_5^5}{y_4^4},\frac{(1+a+b)\sqrt{g(a,b)}}{b}\,\frac{y_5^6}{y_4^5}\right),
\]
with $g$   as in \eqref{gab}.

Finally, the restriction of $\widetilde{\varphi}_2^{(10)}$ to any other $L'_Q$ is given by 
\[
h_{10}(y_4,y_5)=k(a,b,c)\left(\frac{y_5^{10}}{y_4^9},\frac{(1+a)(1+a+b)(1+b)}{ab}\, \frac{y_5^{11}}{y_4^{10}}\right),
\]
with $k(a,b,c)=\frac{(1+a)^2(1+a+b)^3(1+b)^4c^5}{ab^5}$.

 By \cite[Theorem~1]{InHeEs2}  all the three restricted maps $h_2, h_5$ and $h_{10}$ are conjugate to the map
\[
f(x,y)=(y,\xi\frac{y^2}{x})
\]
with $\xi=\phi^2>1$  for $h_2$, $\xi=\left(\frac{(1+a)(1+b)(1+a+b)}{ab}\right)^{\frac{5}{2}}>1$  for $h_5$, and $\xi=\left(\frac{(1+a)(1+b)(1+a+b)}{ab}\right)^{10}>1$ for $h_{10}$.
 By \cite[Lemma 2]{InHeEs2} the maps $h_2, h_5$ and $h_{10}$ have no periodic points and all their components go to infinity. Consequently the map $\widetilde{\varphi}_2$ has no periodic point in any leaf of the form $L'_Q$ and the conclusion follows.
\end{proof}

Finally, we assemble all previous results to describe the dynamics of the  map $\varphi$. 
We denote by $N_{(a,b)}$ and $S_{(a,b,c)}$ the null leaves and the symplectic leaves of the foliations ${\cal F}^\omega$  and ${\cal F}^{P_1}$ respectively:
\begin{align}\label{Nab}
N_{(a,b)} &= \{ {\bf x}\in \Rb_+^7: x_1x_6=a x_3x_4, \, x_2x_7=b x_4x_5\}, \\ \label{Sabc}
S_{(a,b,c)}& = \{ {\bf x}\in N_{(a,b)}: x_1 x_7=cx_4^2\}
\end{align}

The leaves of  ${\cal F}^{P_2}$ are not invariant under any iterate of $\varphi$ and so will be not included in the description of its dynamics. 

\begin{proposition}
\label{proplast}
Consider the cluster map $\varphi:\Rb^7_+\rightarrow \Rb^7_+$  given in \eqref{csomos7}
and its reduced maps $\widehat{\varphi}$ and $\widetilde\varphi_1$ in \eqref{psi16}. 
Let $\phi=\frac{1+\sqrt{5}}{2}$ and consider the null leaves of ${\cal F}^\omega$ and the symplectic leaves of ${\cal F}^{P_1}$ given respectively by \eqref{Nab} and \eqref{Sabc}. 

Then there are four distinct types of orbits of $\varphi$ with respect to the 5-dimensional foliation ${\cal F}^{\omega}$ and its 4-dimensional subfoliation ${\cal F}^{P_1}$. More precisely, a $\varphi$-orbit is either entirely contained in the null leaf $N_{(\phi,\phi)}$, or it circulates between five distinct null leaves: 
\[N_{(a,b)} \rightarrow  N_{\widehat\varphi(a,b)} \rightarrow  N_{\widehat\varphi^{(2)}(a,b)} \rightarrow  N_{\widehat\varphi{(3)}(a,b)}\rightarrow  N_{\widehat\varphi{(4)}(a,b)}.\]
In the null leaf $N_{(\phi,\phi)}$ there exist two exclusive cases:  
\begin{enumerate}
\item the $\varphi$-orbit is entirely contained in the symplectic leaf $S_{(\phi,\phi,\sqrt{\phi^3})}$;
\item  the $\varphi$-orbit circulates between two distinct leaves:
$S_{(\phi,\phi,c)} \rightarrow S_{\widetilde\varphi_1(\phi,\phi,c)},$
 with $c\neq \sqrt{\phi^3}$. 
\end{enumerate}
In any other null leaf $N_{(a,b)}$ there are another two exclusive cases:
\begin{enumerate}
\setcounter{enumi}{2}
\item the $\varphi$-orbit circulates between five distinct symplectic leaves: 
\[S_{Q}\rightarrow  S_{\widetilde\varphi_1(Q)}\rightarrow S_{\widetilde\varphi_1^{(2)}(Q)}\rightarrow S_{\widetilde\varphi_1^{(3)}(Q)}\rightarrow S_{\widetilde\varphi_1^{(4)}(Q)},\]
with $Q=(a,b,\sqrt{g(a,b)})$ and $g(a,b)$ given by \eqref{gab};
\item the $\varphi$-orbit circulates between ten distinct symplectic leaves 
\[S_{Q}\rightarrow S_{\widetilde\varphi_1(Q)}\rightarrow\cdots \rightarrow S_{\widetilde\varphi_1^{(9)}(Q)},\]
with $Q=(a,b,c)$ and $c\neq \sqrt{g(a,b)}$, in the following way: beginning in a leaf $S_{Q}\subset N_{(a,b)}$ the orbit comes back  to $N_{(a,b)}$ to a different leaf $S_{\widetilde\varphi_1^{(5)}(Q)}$ after 5 iterations, and returns to the same leaf  $S_{Q}$ after 10 iterations.
\end{enumerate} 
Moreover, the map $\varphi$ has no  periodic points. 
\end{proposition}

Figure~\ref{figura} below illustrates the contents of the last proposition. To avoid overloading the picture, case  (iii) of the proposition is omitted.

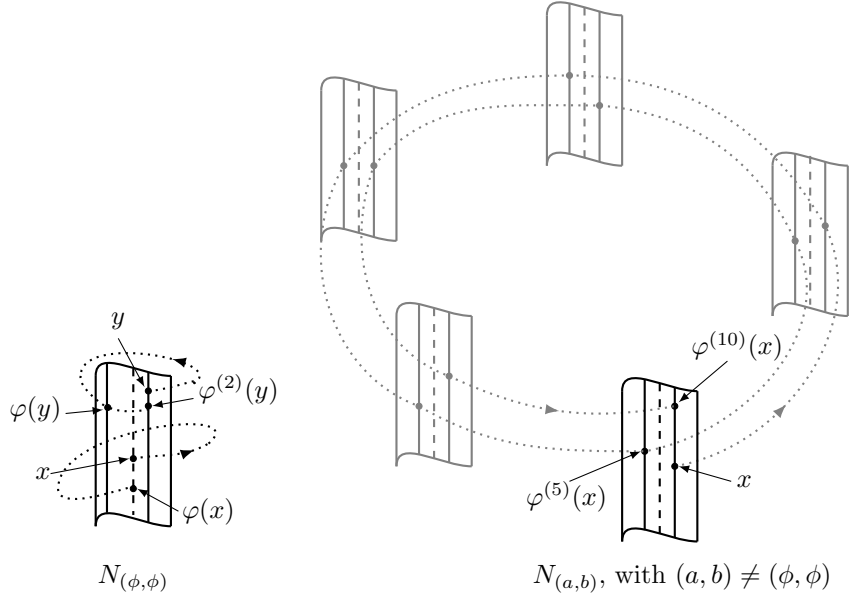
\begin{figure}[ht]
\begin{minipage}{12cm}
\begin{multicols}{2}

\begin{tikzpicture}[-latex]
\tikzset{midarrow/.style={
        decoration={markings,
            mark= at position {#1} with {\arrow{latex}} ,
        },
        postaction={decorate}
    }
}

\draw [-,thick, black] (2,5) to  (2,3);
\draw [-,thick, black] (1,5) to  (1,3);
\draw [-,thick, black] (1,3) to [out=90,in=180] (2,3);
\draw [-,thick, black] (1,5) to [out=90,in=180] (2,5);

\draw[-,thick,black] (1.7,5.05) to  (1.7,3.05);
\draw[-,thick,black] (1.15,3.18) to  (1.15,5.18);
\draw[-,dashed=0.3, thick, black] (1.5,5.1) to  (1.5,3.15);
\draw [fill,black] (1.5,3.5) circle [radius=0.04];
\draw [fill, black] (1.5,3.9) circle [radius=0.04];
\draw [-,midarrow=0.8,thick,black,dotted] (1.5,3.9) to [out=10, in=210] (2.5,4.1);
\draw [-,thick,black,dotted] (2.5,4.1) to [out=40, in=90] (0.5,3.5) to [out=-90, in=30] (1.5,3.5);
\draw [fill] (1.7,4.8) circle [radius=0.04];
\draw [fill] (1.7,4.6) circle [radius=0.04];
\draw [fill] (1.16,4.58) circle [radius=0.04];
\draw [-,thick, black,dotted] (1.7,4.8) to [out=0, in=30] (2.3,4.9);
\draw[-,midarrow=0.3,thick, black,dotted] (2.3,4.9) to [out=30, in=90] (0.8,5);
\draw[-,thick, black,dotted] (0.8,5)  to [out=-100, in=-40] (1.15,4.6) to [out=-40, in=-10] (1.7,4.6);
\node[above] at (2.9,4.5) {$\varphi^{(2)}(y)$};
\node[above] at (1.3,5.5) {$y$};
\draw[-latex] (1.3,5.5) to (1.65,4.85);
\draw[-latex] (2.3,4.7) to  (1.75,4.6);

\node[above] at (0.2,4.25) {$\varphi(y)$};
\draw[-latex] (0.6,4.5) to  (1.1,4.55);
\node[above] at (0.3,3.5) {$x$};
\draw[-latex] (0.4,3.7) to  (1.45,3.9);
\node[above] at (2.5,2.9) {$\varphi(x)$};
\draw[-latex] (2.1,3.2) to  (1.55,3.5);
\node[above] at (1.5,2.0) {$N_{(\phi,\phi)}$};
\end{tikzpicture}
\columnbreak
\begin{tikzpicture}[-latex]
\tikzset{midarrow/.style={
        decoration={markings,
            mark= at position {#1} with {\arrow{latex}} ,
        },
        postaction={decorate}
    }
}

\draw [-,thick, gray] (4,4) to [out=90,in=180] (5,4);
\draw [-,thick, gray] (5,4) to  (5,2);
\draw [-,thick, gray] (4,4) to  (4,2);
\draw [-,thick, gray] (4,2) to [out=90,in=180] (5,2);
\draw[-,thick,gray] (4.7,4.05) to  (4.7,2.05);
\draw[-,thick,gray] (4.3,4.16) to  (4.3,2.16);
\draw[-,dashed=0.3, thick, gray]  (4.5,4.13) to  (4.5,2.15);
\draw [fill,gray] (4.7,3.2) circle [radius=0.04];
\draw [fill,black] (2.3,0.2) circle [radius=0.04];
\draw [fill,gray] (1.7,4.8) circle [radius=0.04];
\draw [fill,gray] (4.3,3) circle [radius=0.04];
\draw [fill,gray] (-1.7,4) circle [radius=0.04];
\draw [fill,gray] (-0.3,1.2) circle [radius=0.04];
\draw [fill,black] (2.7,0.8) circle [radius=0.04];
\draw [-,thick, gray,dotted]  (4.7,3.2) to  [out=120, in=0] (1.3,5.2) to  [out=-180, in=50] (-1.7,4) to  [out=-120, in=-200] (-0.7,0.8) to [out=-30, in=-180] (2.3,0.2)
to [out=10, in=-63] (4.3,3.0) to [out=115, in=-5] (1.7,4.8)  to  [out=-180, in=50] (-1.3,4) to  [out=-110, in=-200] (-0.3,1.2);
\draw [-,midarrow=0.4,thick,gray,dotted] (2.7,0) to [out=10, in=-65] (4.7,3.2);
\draw[-,midarrow=0.5,thick,gray,dotted] (-0.3,1.2) to [out=-25, in=-170] (2.7,0.8);
\draw [-,thick, gray] (1,6) to [out=90,in=180] (2,6);
\draw [-,thick, gray] (2,6) to  (2,4);
\draw [-,thick, gray] (1,6) to  (1,4);
\draw [-,thick, gray] (1,4) to [out=90,in=180] (2,4);
\draw[-,thick,gray] (1.7,4.05) to  (1.7,6.05);
\draw[-,thick,gray] (1.3,4.15) to  (1.3, 6.15);
\draw[-,dashed=0.3, thick, gray]  (1.5,4.1) to  (1.5,6.1);
\draw [fill,gray] (1.3,5.2) circle [radius=0.04];
\draw [-,thick, gray] (-2,5) to [out=90,in=180] (-1,5);
\draw [-,thick, gray] (-1,5) to  (-1,3);
\draw [-,thick, gray] (-2,5) to  (-2,3);
\draw [-,thick, gray] (-2,3) to [out=90,in=180] (-1,3);
\draw[-,dashed=0.3, thick, gray]  (-1.5,3.1) to  (-1.5,5.1);
\draw[-,thick,gray] (-1.3,3.05) to  (-1.3,5.05);
\draw[-,thick,gray] (-1.7,3.15) to  (-1.7,5.15);
\draw [fill,gray] (-1.3,4) circle [radius=0.04];
\draw [-,thick, gray] (-1,2) to [out=90,in=180] (0,2);
\draw [-,thick, gray] (0,2) to  (0,0);
\draw [-,thick, gray] (-1,2) to  (-1,0);
\draw [-,thick, gray] (-1,0) to [out=90,in=180] (0,0);
\draw[-,dashed=0.3, thick, gray]  (-0.5,0.1) to  (-0.5,2.1);
\draw[-,thick,gray] (-0.3,0.05) to  (-0.3,2.05);
\draw[-,thick,gray] (-0.7,0.15) to  (-0.7,2.15);
\draw [fill,gray] (-0.7,0.8) circle [radius=0.04];
\draw [-,thick, black] (2,1) to [out=90,in=180] (3,1);
\draw [-,thick, black] (3,1) to  (3,-1);
\draw [-,thick, black] (2,1) to  (2,-1);
\draw [-,thick, black] (2,-1) to [out=90,in=180] (3,-1);
\draw[-,dashed=0.3, thick, black]  (2.5,-0.9) to  (2.5,1.1);
\draw[-,thick,black] (2.7,-0.95) to  (2.7,1.05);
\draw[-,thick,black] (2.3,-0.85) to  (2.3,1.15);
\draw [fill,black] (2.7,0) circle [radius=0.04];
\draw[-latex] (3.45,-0.2) to  (2.75,0.0);
\node[right] at (3.4,-0.2) {$x$};
\draw[latex-] (2.75,0.85) to  (3.4,1.4);
\node[right] at (2.8,1.6) {$\varphi^{(10)}(x)$};
\node[right] at (0.6,-0.4) {$\varphi^{(5)}(x)$};
\draw[-latex] (1.3,-0.2) to  (2.25,0.2);
\node[above] at (2.8,-1.8) {$N_{(a,b)}$, with $\scriptsize{(a,b)\neq(\phi,\phi)}$};

\end{tikzpicture}
\end{multicols}

\caption{The foliation of the 5D null leaves $N_{(a,b)}$ by 4D symplectic leaves $S_{(a,b,c)}$ in  Example~\ref{ex4}, and the dynamics of $\varphi$ on these leaves. On the left, the $\varphi$-invariant leaf $N_{(\phi,\phi)}$ and its $\varphi^{(2)}$-invariant leaves $S_{(\phi,\phi,c)}$. On the right, a $\varphi^{(5)}$-invariant leaf $N_{(a,b)}$ and its $\varphi^{(10)}$-invariant leaves $S_{(a,b,c)}$. 
\label{figura}}
\end{minipage}
\end{figure}

\end{example}

\bigbreak
\noindent {\bf Acknowledgements}.   

The work of I. Cruz and H. Mena-Matos was partially funded by FCT under the project PEst-C/MAT/UI0144/2013.

The work of  M. E. Sousa-Dias  was partially funded by FCT/Portugal through the projects  UID/MAT/04459/2013 and EXCL/MAT-GEO/0222/2012.

\end{document}